\documentclass[mathpazo]{cicp}
\usepackage{multirow}
\usepackage{mdwlist}
\usepackage{xcolor}
\usepackage{subfigure}
\usepackage{colortbl}

\newtheorem{Algorithm}{Algorithm}

\begin{document}
\title{Parallel-in-Time with Fully Finite Element Multigrid for 2-D Space-fractional Diffusion Equations}


\author[Yue X Q et.~al.]{X.~Q.~Yue\affil{1}, S.~Shu\affil{1}, X.~W.~Xu\affil{2},
                         W.~P.~Bu\affil{1} and K.~J.~Pan\affil{3}\comma\corrauth}
\address{\affilnum{1}\ School of Mathematics and Computational Science,
         Hunan Key Laboratory for Computation and Simulation in Science and Engineering,
         Xiangtan University, Hunan 411105, China \\
         \affilnum{2}\ Institute of Applied Physics and Computational Mathematics,
         Beijing 100088, China \\
         \affilnum{3}\ School of Mathematics and Statistics,
         Central South University, Hunan 410083, China}
\emails{{\tt yuexq@xtu.edu.cn} (X.~Q.~Yue), {\tt shushi@xtu.edu.cn} (S.~Shu),
{\tt xwxu@iapcm.ac.cn} (X.~W.~Xu), {\tt weipingbu@lsec.cc.ac.cn} (W.~P.~Bu),
{\tt pankejia@hotmail.com} (K.~J.~Pan)}

\begin{abstract}
The paper investigates a non-intrusive parallel time integration with multigrid
for space-fractional diffusion equations in two spatial dimensions. We firstly obtain a fully discrete
scheme via using the linear finite element method to discretize spatial and temporal derivatives to propagate solutions.
Next, we present a non-intrusive time-parallelization and its two-level convergence analysis,
where we algorithmically and theoretically generalize the MGRIT to time-dependent fine time-grid propagators. Finally,
numerical illustrations show that the obtained numerical scheme possesses the saturation error order,
theoretical results of the two-level variant deliver good predictions, and significant speedups can be achieved
when compared to parareal and the sequential time-stepping approach.
\end{abstract}


\maketitle

\section{Introduction}
\label{sec1}

In recent years, mathematical models with fractional derivatives and integrals
attract a wide interest of scientists in a variety of fields (including physics, biology and chemistry, etc.),
owing to their potences in descriptions of memory and heredity \cite{k-01}. Particularly,
fractional diffusion equations are shown to afford investigations on subdiffusive phenomena and L\'{e}vy fights \cite{m-01}.
In this article, we are interested in a class of two-dimensional space-fractional diffusion equations (SFDEs)
\begin{align}\label{chp-01-01}
  &\frac{\partial u(x,y,t)}{\partial t}=K_x\frac{\partial^{2\beta}u(x,y,t)}{\partial|x|^{2\beta}}
  +K_y\frac{\partial^{2\gamma}u(x,y,t)}{\partial|y|^{2\gamma}}+f(x,y,t),~t\in I=(0,T],~(x,y)\in\Omega\\
  &\label{chp-01-02}u(x,y,t)=0,~t\in I,~(x,y)\in\partial\Omega\\
  &\label{chp-01-03}u(x,y,0)=\psi_0(x,y),~(x,y)\in\Omega
\end{align}
with orders $1/2<\beta$, $\gamma<1$, constants $K_x$, $K_y>0$, solution domain $\Omega=(a,b)\times(c,d)$,
and Riesz fractional derivatives
\begin{eqnarray*}
  \frac{\partial^{2\beta}u}{\partial|x|^{2\beta}}=-\frac{1}{2\cos(\beta\pi)}({}_xD_L^{2\beta}u + {}_xD_R^{2\beta}u),~~
  \frac{\partial^{2\gamma}u}{\partial|y|^{2\gamma}}=-\frac{1}{2\cos(\gamma\pi)}({}_yD_L^{2\gamma}u + {}_yD_R^{2\gamma}u),
\end{eqnarray*}
where
\begin{align*}
  &{}_xD_L^{2\beta}u = \frac{1}{\Gamma(2-2\beta)}\frac{\partial^2}{\partial x^2}\int_a^x(x-s)^{1-2\beta}u(s,y,t)ds,~~
  {}_xD_R^{2\beta}u = \frac{1}{\Gamma(2-2\beta)}\frac{\partial^2}{\partial x^2}\int_x^b(s-x)^{1-2\beta}uds,\\
  &{}_yD_L^{2\gamma}u = \frac{1}{\Gamma(2-2\gamma)}\frac{\partial^2}{\partial y^2}\int_c^y(y-r)^{1-2\gamma}u(x,r,t)dr,~~
  {}_yD_R^{2\gamma}u = \frac{1}{\Gamma(2-2\gamma)}\frac{\partial^2}{\partial y^2}\int_y^d(r-y)^{1-2\gamma}udr.
\end{align*}

Since closed-form analytical solutions of fractional models are rarely accessible in practice,
the numerical solutions become very prevalent to empower their successful applications.
In literatures, numerical methods of SFDEs proposed to achieve high accuracy and efficiency include
finite difference \cite{w-01,c-01,h-01,c-02,c-03,l-01,l-04,h-04}, finite element (FE) \cite{b-02,d-01,f-01,y-03,y-01,y-02},
finite volume \cite{h-02,j-01,k-02} and spectral (element) \cite{s-01,b-01,p-01} methods.
It must be emphasized that no matter which discretization is applied,
there usually persists intensive computational task in nonlocality caused by fractional differential operators \cite{g-03}.
Numerous scholars are working to identify fast algorithms most appropriate to tackle this challenge,
see \cite{p-03,p-02,g-01,j-02,s-02,c-04} and related references therein. Except for these fast solutions, 
parallel computing can be viewed as another potential technique or even a basic strategy.
Gong et al. presented MPI-based and GPU-based parallel algorithms for one-dimensional Riesz SFDEs \cite{g-04,w-03},
whose speedups are both achieved by spatial parallelism with sequential time-stepping approach,
using some time propagator to integrate from one time to the next.
However, future computing speed must rely on the increased concurrency provided by more, instead of faster, processors.
An immediate consequence of this is that solution algorithms, limited to spatial parallelism, for problems with evolutionary behavior
entail long overall computation time, often exceeding computing resources available to resolve multidimensional SFDEs. Thus,
algorithms achieving parallelism in time have been of especially high demand over the past decade. Currently, parareal in time \cite{l-02}
and multigrid-reduction-in-time (MGRIT) \cite{f-02} are two active choices. Wu et al. analyzed convergence properties of the parareal algorithm
for SFDEs via certain underlying ODEs in constant time-steps \cite{w-04,w-02}, but they uninvolved large-scale testings.
Observe that parareal can be interpreted as a two-level multigrid (reduction) method \cite{f-02,g-02},
its concurrency is still limited because of the large sequential coarse-grid solve.
MGRIT, a non-intrusively and truly multilevel algorithm implemented in the open-source XBraid \cite{x-01} with optimal parallel communication,
counteracts this and enables us to approximate simultaneously the evolution over all time points.
The non-intrusive nature of MGRIT relies upon modalities of fine and coarse time-grid propagators rather than their internals.
It has been proven to be effective and analyzed sharply in the two-level setting
for integer order basic parabolic and hyperbolic problems \cite{d-02}, 
however, with limitations on analysis that they only consider linear problems and the temporal grid is uniform, i.e.,
fine time-grid propagators are all the same. Furthermore, 
from the survey of references, there are no calculations of the MGRIT algorithm to SFDEs,
nor to FE discretizations of parabolic and hyperbolic problems in time.

%
%

The main aim of the paper is to propose and analyze a non-intrusive optimal-scaling MGRIT algorithm for
space-time FE discretizations of problem \eqref{chp-01-01}-\eqref{chp-01-03}
in uniform and nonuniform temporal partitions, where we shall extend the scope of the MGRIT method
and develop a library of MGRIT modifications to time-dependent propagators.
The outline of our presentation proceeds as follows. In Section 2, we derive a fully discrete FE scheme.
Section 3 introduces the MGRIT solver followed by its two-level convergence analysis to the obtained scheme.
In Section 4, we report some numerical results to illustrate optimal convergence rates in both space and time,
present theoretical confirmations and analyze weak and strong scaling studies to show benefits.
Finally, relevant results are summarized and follow-up work are drawn in Section 5.


\section{Space-time FE discretization for SFDEs}
\label{sec2}

This section deals with the construction of our space-time FE scheme for SFDEs.
Here we denote by $(\cdot,\cdot)_{L^2(\Omega)}$ and $\|\cdot\|_{L^2(\Omega)}$ the inner product and its associated norm on $L^2(\Omega)$.
First, we introduce some fractional derivative spaces.


\begin{definition} (Left and right fractional derivative spaces)
  For a given constant $\mu>0$, define norms
\begin{align*}
  &\|u\|_{J_L^\mu(\Omega)} := (\|u\|^2_{L^2(\Omega)}+\|{}_xD_L^\mu u\|^2_{L^2(\Omega)}+\|{}_yD_L^\mu u\|^2_{L^2(\Omega)})^{\frac{1}{2}},\\
  &\|u\|_{J_R^\mu(\Omega)} := (\|u\|^2_{L^2(\Omega)}+\|{}_xD_R^\mu u\|^2_{L^2(\Omega)}+\|{}_yD_R^\mu u\|^2_{L^2(\Omega)})^{\frac{1}{2}}.
\end{align*}
  Let $J_{L,0}^\mu(\Omega)$ and $J_{R,0}^\mu(\Omega)$ be closures of $C_0^\infty(\Omega)$
  in regard to $\|\cdot\|_{J_L^\mu(\Omega)}$ and $\|\cdot\|_{J_R^\mu(\Omega)}$, respectively.
\end{definition}

\begin{definition} (Fractional Sobolev space)
  For a given constant $\mu>0$, define the norm
\begin{eqnarray*}
  \|u\|_{H^\mu(\Omega)} := \Big{[}\int_\Omega(1+|\xi|^2)^\mu|\tilde{u}(\xi)|^2d\xi\Big{]}^{\frac{1}{2}},
\end{eqnarray*}
  where $\tilde{u}$ is the Fourier transform of $u$.
  Let $H_0^\mu(\Omega)$ be the closure of $C_0^\infty(\Omega)$ in $\|\cdot\|_{H^\mu(\Omega)}$ norm sense.
\end{definition}

\begin{remark}
  The mathematical equivalence among $J_{L,0}^\mu(\Omega)$, $J_{R,0}^\mu(\Omega)$ and $H_0^\mu(\Omega)$ with related norms
  can be proved to the case where $\mu\in(1/2,1)$, analogously to the work \cite{e-01}.
\end{remark}



%

Utilizing Lemma 5 in \cite{b-02}, the weak formulation of problem \eqref{chp-01-01}-\eqref{chp-01-03} is derived in reference to $t>0$:
given $\psi_0\in L^2(\Omega)$, $g\in L^2(\Omega,I)$ and $Q_t:=\Omega\times(0,t)$,
hunting for $u\in \mathcal{G}:=[H_0^\beta(\Omega)\cap H_0^\gamma(\Omega)]\times H^1(I)$ subject to $u(x,y,0)=\psi_0(x,y)$ and
\begin{eqnarray}\label{chp-02-01}
  \int_0^t(\frac{\partial u}{\partial \varrho},v)_{L^2(\Omega)}d\varrho+B^t_\Omega(u,v)=\int_0^t(f,v)_{L^2(\Omega)}d\varrho,~
  \forall v\in \mathcal{G}^*:=[H_0^\beta(\Omega)\cap H_0^\gamma(\Omega)]\times L^2(I),
\end{eqnarray}
where
\begin{align}\label{chp-02-07}
  &B^t_\Omega(u,v)=\int_0^t\frac{K_x}{2\cos(\beta\pi)}\Big{[}({}_xD^\beta_Lu,{}_xD^\beta_Rv)_{L^2(\Omega)}
  +({}_xD^\beta_Ru,{}_xD^\beta_Lv)_{L^2(\Omega)}\Big{]}d\varrho + \notag \\
  &\qquad\qquad\qquad\qquad\int_0^t\frac{K_y}{2\cos(\gamma\pi)}\Big{[}({}_yD^\gamma_Lu,{}_yD^\gamma_Rv)_{L^2(\Omega)}
  +({}_yD^\gamma_Ru,{}_yD^\gamma_Lv)_{L^2(\Omega)}\Big{]}d\varrho.
\end{align}

For the depiction of our space-time FE discretization, we define a (possibly nonuniform) temporal partition $0=t_0<t_1<\cdots<t_N=T$
and a uniform spatial triangulation $\mathcal{T}_h$ with constant spacings $h_x=(b-a)/M_\beta$ and $h_y=(d-c)/M_\gamma$,
let $\mathcal{K}_h$ be the set of interior mesh points in $\mathcal{T}_h$, denote the $j$-th subintervals $I_j=(t_{j-1},t_j)$ and
$\tilde{I}_j=(0,t_j)$ for $1\le j\le N$, and $\Omega_h=\{e_h:e_h\in\mathcal{T}_h\}$. 
We choose the usual spaces in tensor products
\begin{eqnarray*}
  \mathcal{G}_n=\mathcal{V}_h(\Omega_h)\times\mathcal{V}_\tau(\tilde{I}_n),~~
  \mathcal{G}^*_n=\mathcal{V}_h(\Omega_h)\times\mathcal{V}^*_\tau(I_n),
\end{eqnarray*}
where $\mathcal{V}_h(\Omega_h)=\{w_h\in H_0^\beta(\Omega)\cap H_0^\gamma(\Omega)\cap C(\overline{\Omega}):
w_h|_{e_h}\in\mathcal{P}_1(e_h),~\forall e_h\in\mathcal{T}_h\}$,
$\mathcal{V}_\tau(\tilde{I}_n) = \{v_\tau\in \mathcal{C}(\overline{\tilde{I}_n}):
v_\tau(0)=1,~v_\tau(t)|_{I_j}\in\mathcal{P}_1(I_j),~j=1,\cdots, n\}$ and
$\mathcal{V}^*_\tau(I_n) = \{v_\tau\in L^2(I_n):v_\tau(t)|_{I_n}\in\mathcal{P}_0(I_n)\}$
with $\mathcal{P}_k$ as the space of all polynomials of degree $\leq k$.

Now a fully discrete FE approximation for \eqref{chp-02-01} is singled out immediately: for $Q_n:=\Omega_h\times I_n$,
to seek the solution $u_{h\tau}\in \mathcal{G}_n$ satisfying $u_{h\tau}(x,y,0)=\psi^I_0(x,y)$ and
\begin{eqnarray}\label{chp-02-02}
  \int_{t_{n-1}}^{t_n}(\frac{\partial u_{h\tau}}{\partial t},v_{h\tau})_{L^2(\Omega)}dt+B^n_\Omega(u_{h\tau},v_{h\tau})
  =\int_{t_{n-1}}^{t_n}(f,v_{h\tau})_{L^2(\Omega)}dt,~\forall v_{h\tau}\in \mathcal{G}^*_n,
\end{eqnarray}
where the function $\psi^I_0(x,y)\in\mathcal{G}_n$ approximates to the initial $\psi_0(x,y)\in L^2(\Omega)$, and
\begin{align*}
  & B^n_\Omega(u_{h\tau},v_{h\tau})
  = \int_{t_{n-1}}^{t_n}\frac{K_x}{2\cos(\beta\pi)}\Big{[}({}_xD^\beta_Lu_{h\tau},{}_xD^\beta_Rv_{h\tau})_{L^2(\Omega)}
  +({}_xD^\beta_Ru_{h\tau},{}_xD^\beta_Lv_{h\tau})_{L^2(\Omega)}\Big{]}dt + \notag \\
  &\qquad\qquad\qquad\qquad\int_{t_{n-1}}^{t_n}\frac{K_y}{2\cos(\gamma\pi)}\Big{[}({}_yD^\gamma_Lu_{h\tau},{}_yD^\gamma_Rv_{h\tau})_{L^2(\Omega)}
  +({}_yD^\gamma_Ru_{h\tau},{}_yD^\gamma_Lv_{h\tau})_{L^2(\Omega)}\Big{]}dt.
\end{align*}

Note that $\mathcal{G}^*_n$ and $u_{h\tau}$ can be respectively written as $\mathcal{G}^*_n=\textsf{span}\{\phi_l(x,y)\times 1,~l=1,\cdots,|\mathcal{K}_h|\}$
and $u_{h\tau}(x,y,t)=\sum_{k=0}^nu_h^k(x,y)\mathcal{L}_k(t)$, with $\phi_l(x,y)$ being the Lagrange linear shape function at point
$P_l(x_l,y_l)\in \mathcal{K}_h$, as well as
\begin{eqnarray*}
  \mathcal{L}_0(t)=\left \{
    \begin{aligned}
    & \frac{t_1-t}{t_1-t_0},~t \in I_1 \\
    & 0,~t\in \tilde{I}_n\setminus I_1
    \end{aligned}
  \right.,~~
  \mathcal{L}_k(t)=\left \{
    \begin{aligned}
    & \frac{t_{k+1}-t}{t_{k+1}-t_k},~t \in I_{k+1} \\
    & \frac{t-t_{k-1}}{t_k-t_{k-1}},~t \in I_k \\
    & 0,~t\in \tilde{I}_n\setminus (I_k \cup I_{k+1})
    \end{aligned}
  \right.,~~
  \mathcal{L}_n(t)=\left \{
    \begin{aligned}
    & \frac{t-t_{n-1}}{t_n-t_{n-1}},~t \in I_n \\
    & 0,~t\in \tilde{I}_n\setminus I_n
    \end{aligned}
  \right..
\end{eqnarray*}
By simple algebraic calculations, we deduce
\begin{eqnarray*}
  (\mathcal{L}_n,1)_{L^2(I_n)}=(\mathcal{L}_{n-1},1)_{L^2(I_n)}=\frac{t_n-t_{n-1}}{2};~~~
  (\mathcal{L}_k,1)_{L^2(I_n)}=0,~k<n-1.
\end{eqnarray*}
Hence, to compute $\mathcal{U}_n=(u^n_1,u^n_2,\cdots,u^n_{|\mathcal{K}_h|})^T$ satisfying Eq. \eqref{chp-02-02} we solve
\begin{eqnarray}\label{chp-02-03}
  \Big{[}\mathcal{M}_h+\frac{t_n-t_{n-1}}{2}(K_x\mathcal{A}^\beta_x+K_y\mathcal{A}^\gamma_y)\Big{]}\mathcal{U}_n=\mathcal{F}_n
  +\Big{[}\mathcal{M}_h-\frac{t_n-t_{n-1}}{2}(K_x\mathcal{A}^\beta_x+K_y\mathcal{A}^\gamma_y)\Big{]}\mathcal{U}_{n-1},
\end{eqnarray}
where initial guess $\mathcal{U}_0=(u^0_1,u^0_2,\cdots,u^0_{|\mathcal{K}_h|})^T$ with $u^0_i=\psi^I_0(x_i,y_i)$,
vector $\mathcal{F}_n=(f^n_1,f^n_2,\cdots,f^n_{|\mathcal{K}_h|})^T$ with $f^n_j=\int_{t_{n-1}}^{t_n}(f,\phi_j)_{L^2(\Omega)}dt$,
mass matrix $\mathcal{M}_h$ takes the block-tridiagonal form
\begin{eqnarray*}\label{chp-02-04}
  \frac{h_xh_y}{12}\left(
   \begin{array}{ccccc}
    \mathcal{M}_1 & \mathcal{M}_2 &   &  &  \\
    \mathcal{M}^T_2 & \mathcal{M}_1 & \mathcal{M}_2 &  &  \\
      & \ddots & \ddots & \ddots &  \\
      &  & \mathcal{M}^T_2 & \mathcal{M}_1 & \mathcal{M}_2 \\
      &  &   & \mathcal{M}^T_2 & \mathcal{M}_1
   \end{array}
  \right)_{(M_\gamma-1)}
\end{eqnarray*}
with
\begin{eqnarray*}
  \mathcal{M}_1 = \left(
   \begin{array}{ccccc}
    6 & 1 &   &  &  \\
    1 & 6 & 1 &  &  \\
      & \ddots & \ddots & \ddots &  \\
      &  & 1 & 6 & 1 \\
      &  &   & 1 & 6
   \end{array}
  \right)_{(M_\beta-1)}~~\mbox{and}~~
  \mathcal{M}_2 = \left(
   \begin{array}{ccccc}
    1 & 1 &   &  &  \\
      & 1 & 1 &  &  \\
      &  & \ddots & \ddots &  \\
      &  &   & 1 & 1 \\
      &  &   &   & 1
   \end{array}
  \right)_{(M_\beta-1)},
\end{eqnarray*}
stiffness matrices
\begin{eqnarray*}\label{chp-02-05}
  \mathcal{A}^\beta_x=\frac{h^{1-2\beta}_xh_y}{2\cos(\beta\pi)\Gamma(5-2\beta)}\left(
   \begin{array}{ccccc}
    \mathcal{A}_1^\beta & \mathcal{A}_2^\beta &   &  &  \\
    (\mathcal{A}_2^\beta)^T & \mathcal{A}_1^\beta & \mathcal{A}_2^\beta &  &  \\
      & \ddots & \ddots & \ddots &  \\
      &  & (\mathcal{A}_2^\beta)^T & \mathcal{A}_1^\beta & \mathcal{A}_2^\beta \\
      &  &   & (\mathcal{A}_2^\beta)^T & \mathcal{A}_1^\beta
   \end{array}
  \right)_{(M_\gamma-1)}
\end{eqnarray*}
and
\begin{eqnarray}\label{chp-02-06}
  \mathcal{A}^\gamma_y=\frac{h_xh^{1-2\gamma}_y}{2\cos(\gamma\pi)\Gamma(5-2\gamma)}\mathcal{P}^T_\pi\left(
   \begin{array}{ccccc}
    \mathcal{A}_1^\gamma & \mathcal{A}_2^\gamma &   &  &  \\
    (\mathcal{A}_2^\gamma)^T & \mathcal{A}_1^\gamma & \mathcal{A}_2^\gamma &  &  \\
      & \ddots & \ddots & \ddots &  \\
      &  & (\mathcal{A}_2^\gamma)^T & \mathcal{A}_1^\gamma & \mathcal{A}_2^\gamma \\
      &  &   & (\mathcal{A}_2^\gamma)^T & \mathcal{A}_1^\gamma
   \end{array}
  \right)_{(M_\beta-1)}\mathcal{P}_\pi
\end{eqnarray}
with the permutation matrix $\mathcal{P}_\pi$ produced in terms of the identity matrix by 
\begin{eqnarray*}
  \pi=\Big{\{}1,M_\beta,\cdots,1+(M_\beta-1)(M_\gamma-2);~\cdots;~M_\beta-1,2(M_\beta-1),\cdots,(M_\beta-1)(M_\gamma-1)\Big{\}},
\end{eqnarray*}
Toeplitz matrices $\mathcal{A}_l^\varrho=(a_{i,j}^{l,\varrho})_{(M_\varrho-1)\times(M_\varrho-1)}$
($\varrho=\beta,\gamma$) whose entries are in forms 
\begin{align*}
  \left \{
    \begin{aligned}
      & a^{1,\varrho}_{i,i}=2^{6-2\varrho}+16\varrho-40,\qquad\qquad
      \quad\qquad\qquad\qquad\qquad\quad ~~~i=1,\cdots,M_\varrho-1\\
      & a^{1,\varrho}_{j,j+1}=a^{1,\varrho}_{j+1,j}=2\cdot3^{4-2\varrho}+(2\varrho-6)2^{5-2\varrho}-16\varrho+34,
      \quad\qquad ~j=1,\cdots,M_\varrho-2\\
      & a^{1,\varrho}_{k,k+l}= a^{1,\varrho}_{k+l,k} = 4(4-2\varrho)[-(l-1)^{3-2\varrho}+2l^{3-2\varrho}-(l+1)^{3-2\varrho}]-2(l-2)^{4-2\varrho} \\
      & ~~~+4(l-1)^{4-2\varrho}-4(l+1)^{4-2\varrho}+2(l+2)^{4-2\varrho},\quad k=1,\cdots,M_\varrho-l-1;~l=2,\cdots,M_\varrho-2
    \end{aligned}
  \right.
\end{align*}
and
\begin{align*}
  \left \{
    \begin{aligned}
      & a^{2,\varrho}_{i,i}=4-2^{4-2\varrho}\varrho,\qquad\qquad\qquad\qquad\quad\qquad\qquad\qquad\qquad\qquad ~~~i=1,\cdots,M_\varrho-1\\
      & a^{2,\varrho}_{j,j+1}=4-2^{4-2\varrho}\varrho,\qquad\qquad\quad\qquad\qquad\qquad\qquad\qquad\qquad\qquad j=1,\cdots,M_\varrho-2\\
      & a^{2,\varrho}_{k,k+l}=(4-2\varrho)[(l-2)^{3-2\varrho}-(l-1)^{3-2\varrho}-l^{3-2\varrho}+(l+1)^{3-2\varrho}]+2(l-2)^{4-2\varrho}\\
      & ~~~-6(l-1)^{4-2\varrho}+6l^{4-2\varrho}-2(l+1)^{4-2\varrho},\quad k=1,\cdots,M_\varrho-l-1;~l=2,\cdots,M_\varrho-2\\
      & a^{2,\varrho}_{p+m,p}=(4-2\varrho)[(m-1)^{3-2\varrho}-m^{3-2\varrho}-(m+1)^{3-2\varrho}+(m+2)^{3-2\varrho}]+2(m-1)^{4-2\varrho}\\
      & ~~~-6m^{4-2\varrho}+6(m+1)^{4-2\varrho}-2(m+2)^{4-2\varrho},\quad p=1,\cdots,M_\varrho-m-1;~m=1,\cdots,M_\varrho-2
    \end{aligned}
  \right..
\end{align*}

\begin{remark}
  It is worthwhile to point out that $\mathcal{P}_\pi$ is introduced for a more economical memory requirement,
  since $\mathcal{A}^\gamma_y$ is naturally of the full block form,
  while \eqref{chp-02-06} computes $\mathcal{A}^\gamma_y$ without explicit generation of $\mathcal{P}_\pi$ or $\mathcal{P}^T_\pi$.
\end{remark}


\section{Time-parallelization and its two-level convergence analysis}
\label{sec3}

This section is devoted to the time-parallelization for the forward time-marching loop \eqref{chp-02-03}
as well as the two-level convergence analysis.

\subsection{The MGRIT algorithm}

Consider the forward problem \eqref{chp-02-03}, a one-step method of Eq. \eqref{chp-01-01},
which is equivalent to the block unit lower bidiagonal system
\begin{equation}\label{ode-03}
\mathcal{A}\mathcal{U}:=
\begin{bmatrix}
  I &        &        &   \\
  -\Psi_1 & I &        &   \\
        & \ddots & \ddots &   \\
        &        &  -\Psi_N & I
\end{bmatrix}
\begin{bmatrix}
\mathcal{U}_0 \\
\mathcal{U}_1 \\
\vdots \\
\mathcal{U}_N
\end{bmatrix}
=
\begin{bmatrix}
\mathcal{G}_0 \\
\mathcal{G}_1 \\
\vdots \\
\mathcal{G}_N
\end{bmatrix}:=\mathcal{G},
\end{equation}
where the $n$-th time-grid propagator
\begin{eqnarray}\label{equ-03-01}
  \Psi_n = \Big{[}\mathcal{M}_h+\frac{t_n-t_{n-1}}{2}(K_x\mathcal{A}^\beta_x+K_y\mathcal{A}^\gamma_y)\Big{]}^{-1}
  \Big{[}\mathcal{M}_h-\frac{t_n-t_{n-1}}{2}(K_x\mathcal{A}^\beta_x+K_y\mathcal{A}^\gamma_y)\Big{]}
\end{eqnarray}
and
\begin{eqnarray*}
  \mathcal{G}_0=\mathcal{U}_0,~~\mathcal{G}_n = \Big{[}\mathcal{M}_h+\frac{t_n-t_{n-1}}{2}(K_x\mathcal{A}^\beta_x
  +K_y\mathcal{A}^\gamma_y)\Big{]}^{-1}\mathcal{F}_n,~n=1,\cdots,N.
\end{eqnarray*}
Obviously, the inverse in Eq. \eqref{equ-03-01} corresponds to a spatial solve.

\begin{figure}[htp]
\centerline{
\includegraphics[scale=0.3]{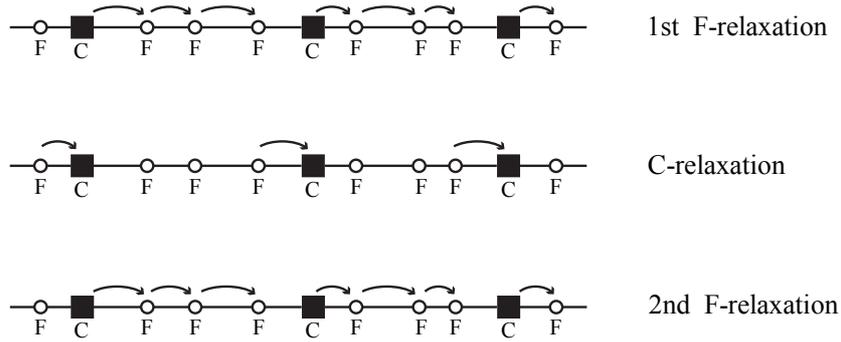}}
\caption{Schematic of the update sequence during FCF-relaxation with the coarsening factor $m=4$.}\label{fig-03-01}
\end{figure}

Regarding the MGRIT algorithm to solve the global space-time problem \eqref{ode-03}, various components have to be chosen.
Let $m$ be the coarsening factor in time and $N_c=N/m$, we define a coarse mesh
$\tilde{t}_i=t_{mi}$, $i=0,\cdots,N_c$. 
In this setting, all $\tilde{t}_i$ are C-points and the others are F-points.
FCF-relaxation depicted in Fig. \ref{fig-03-01}, an initial F-relaxation followed by a C-relaxation and then a second F-relaxation,
is often the most reliable choice to produce optimal and scalable multilevel iterations \cite{f-02}.
Define the injection at C-points as our restriction operator, 
and the injection followed by an F-relaxation over the fine-grid operator as our interpolation operator.
The multilevel hierarchy can be constructed by applying the above processes recursively.
Both sequential time-stepping and MGRIT are $\mathcal{O}(N)$ in terms of spatial solve, but MGRIT is highly concurrent.
About $\nu_t[2m/(m-1)+1]$ more processors are actually required in temporal concurrency to outweigh the extra work of MGRIT,
where $\nu_t$ is the number of MGRIT iterations \cite{f-02}.
Below is the MGRIT V-cycle algorithm, where $\mathcal{A}^{(0)}=\mathcal{A}$, 
$\mathcal{G}^{(0)}=\mathcal{G}$, $L=\log_mN$, $\mathcal{A}^{(l+1)}$, $\mathcal{R}^{(l)}$ and $\mathcal{P}^{(l)}$
($l=0,\cdots,L-1$) correspond to the $l$-th coarse-scale time re-discretization, restriction and interpolation, respectively.


\begin{Algorithm}\label{alg-03-01}
MGRIT algorithm with V-cycle: $\mathcal{U}^{(l)}=\textrm{MGRIT}(\mathcal{A}^{(l)},\mathcal{U}^{(l)},\mathcal{G}^{(l)})$.
\begin{description}

  \item[Step 1] Apply FCF-relaxation to $\mathcal{A}^{(l)} \mathcal{U}^{(l)} = \mathcal{G}^{(l)}$.

  \item[Step 2] Restrict the residual $\mathcal{G}^{(l+1)}=\mathcal{R}^{(l)}(\mathcal{G}^{(l)} - \mathcal{A}^{(l)} \mathcal{U}^{(l)})$.

  \item[Step 3] If $l+1=L$, then solve $\mathcal{A}^{(L)}\mathcal{U}^{(L)}=\mathcal{G}^{(L)}$; \\ Else perform
  $\mathcal{U}^{(l+1)}=\textrm{MGRIT}(\mathcal{A}^{(l+1)},0,\mathcal{G}^{(l+1)})$.

  \item[Step 4] Do the coarse-grid correction $\mathcal{U}^{(l)} = \mathcal{U}^{(l)} + \mathcal{P}^{(l)} \mathcal{U}^{(l+1)}$.

\end{description}
\end{Algorithm}

\begin{remark}
  To save computational work, Step 4 of Algorithm \ref{alg-03-01} is done by just correcting C-point values,
  and updating F-point values only when the Euclidean norm of the residual is small enough.
  The reason is that the correction at F-points is equivalent to an F-relaxation,
  which will be performed in Step 1 of the subsequent iteration.
\end{remark}

\subsection{Implementation details}

For numerical experiments, the general-purpose parallel-in-time library XBraid \cite{x-01} was employed.
Eq. \eqref{ode-03} by uniform temporal partitions can be solved by XBraid in a straightforward way,
whereas modifications on XBraid was done to time-dependent propagators as part of this study.
Wrapper routines were written in C and integrals were calculated by a quadrature formula.

One of the most noteworthy is that we split processors and communicators into spatial and temporal groups
for purpose of running parallelized modules in space, time or both.
In view of the linearity of Eq. \eqref{chp-01-01} and Toeplitz structures of $\mathcal{M}_h$,
$\mathcal{A}^\beta_x$ and $\mathcal{A}^\gamma_y$, all submatrices $\mathcal{M}_l$ and $\mathcal{A}_l^\varrho$
($l=1,2$; $\varrho=\beta,\gamma$) are only set up once to reduce computational work.
The space-time approximation is obtained until the space-time residual norm in the discrete $L^2$ sense
is less than the absolute halting tolerance $10^{-9}$,
where all spatial solves are accomplished by the Conjugate Gradient algorithm provided in the HYPRE 
library \cite{h-03} with $10^{-9}$ as the relative tolerance for stopping and the Euclidean norm used to measure solution progress.
Since the procedure involves only matrix-by-vector multiplications, matrices $\mathcal{M}_h$,
$\mathcal{A}^\beta_x$ and $\mathcal{A}^\gamma_y$ are kept in unassembled form to save on memory.
We skip any work on the first MGRIT down cycle.
In addition, we choose random initial guesses for the entire temporal grid hierarchy,
except that the initial condition \eqref{chp-01-03} is used at $t=0$ on the finest grid.

\subsection{Two-level convergence analysis}

Setting $L=2$, Algorithm \ref{alg-03-01} in this case reduces to a two-level scheme,
whose error propagator is characterized in the following lemma.

\begin{lemma}\label{lem-03-01}
  Let $\mathcal{E}$ be the error of \eqref{ode-03} and $m$ the coarsening factor.
  Define $\mathcal{B}_{\Psi,s}^l=\Psi_s\Psi_{s-1}\cdots\Psi_{s-l+1}$ with $\mathcal{B}_{\Psi,s}^0=I$.
  Then, after an iteration of the two-level version of Algorithm \ref{alg-03-01},
  the new error $\tilde{\mathcal{E}}$ at C-points satisfies
\begin{eqnarray*}
  \tilde{\mathcal{E}}_0=\tilde{\mathcal{E}}_m=0,~~\tilde{\mathcal{E}}_{km}=\sum_{i=0}^{k-2}\mathcal{B}_{\Psi^\Delta,k}^{k-2-i}
  (\mathcal{B}_{\Psi,(i+2)m}^m-\Psi^\Delta_{i+2})\mathcal{B}_{\Psi,(i+1)m}^m\mathcal{E}_{im},~k=2,\cdots,N_c,
\end{eqnarray*}
  where $\Psi^\Delta_j$ is the coarse time-grid propagator at the coarse time-scale point $\tilde{t}_j$.
\end{lemma}

\begin{proof}
  Let $\mathcal{V}=\mathcal{U}-\mathcal{E}$ be the approximation, 
  we have the update sequence during FCF-relaxation
\begin{eqnarray*}
  \mathcal{V}_{km-1}=\mathcal{B}_{\Psi,km-1}^{m-1}\mathcal{V}_{(k-1)m}+\sum_{i=1}^{m-1}\mathcal{B}_{\Psi,km-1}^{i-1}\mathcal{G}_{km-i},~k=1,\cdots,N_c;
  \quad \mbox{(the initial F-relaxation)} \\
  \mathcal{V}_0=\mathcal{G}_0,~\mathcal{V}_{km}=\mathcal{B}_{\Psi,km}^m\mathcal{V}_{(k-1)m}+\sum_{i=0}^{m-1}\mathcal{B}_{\Psi,km}^i\mathcal{G}_{km-i},
  ~k=1,\cdots,N_c;\quad \mbox{(the C-relaxation)} \\
  \left \{
    \begin{aligned}
    & \mathcal{V}_{m-1}=\mathcal{B}_{\Psi,m-1}^{m-1}\mathcal{G}_0+\sum_{i=1}^{m-1}\mathcal{B}_{\Psi,m-1}^{i-1}\mathcal{G}_{m-i} \\
    & \mathcal{V}_{km-1}=\mathcal{B}_{\Psi,km-1}^{2m-1}\mathcal{V}_{(k-2)m}+\\
    & \quad \sum_{i=0}^{m-1}\mathcal{B}_{\Psi,km-1}^{m-1+i} \mathcal{G}_{(k-1)m-i}
    +\sum_{i=1}^{m-1}\mathcal{B}_{\Psi,km-1}^{i-1}\mathcal{G}_{km-i},~k=2,\cdots,N_c
    \end{aligned}
  \right.. \quad \mbox{(the second F-relaxation)}
\end{eqnarray*}
  Notice that the exact solution $\mathcal{U}$ can be written in the form
\begin{eqnarray*}
  \sum_{i=0}^{m-1}\mathcal{B}_{\Psi,(k-1)m}^i\mathcal{G}_{(k-1)m-i}=\mathcal{U}_{(k-1)m}-\mathcal{B}_{\Psi,(k-1)m}^m\mathcal{U}_{(k-2)m},~k=2,\cdots,N_c,
\end{eqnarray*}
  which follows from the recursion \eqref{chp-02-03}. Hence, the residual at the C-points becomes
\begin{eqnarray*}
  \mathcal{G}^{(2)}_0=0,~~\mathcal{G}^{(2)}_1=\mathcal{B}_{\Psi,m}^m\mathcal{E}_0,~~
  \mathcal{G}^{(2)}_k=\mathcal{B}_{\Psi,km}^m(\mathcal{E}_{(k-1)m}-\mathcal{B}_{\Psi,(k-1)m}^m\mathcal{E}_{(k-2)m}),~k=2,\cdots,N_c.
\end{eqnarray*}
  Then we can get the coarse-grid solution
\begin{eqnarray*}
  \mathcal{U}^{(2)}_0=0,~~\mathcal{U}^{(2)}_k=\Psi^\Delta_k\mathcal{U}^{(2)}_{k-1}+\mathcal{G}^{(2)}_k,~k=1,\cdots,N_c.
\end{eqnarray*}
  which gives
\begin{align}\label{equ-03-09}
  \left \{
    \begin{aligned}
      & \mathcal{U}^{(2)}_1=\mathcal{G}^{(2)}_1\\
      & \mathcal{U}^{(2)}_k=\mathcal{B}_{\Psi,km}^m\mathcal{E}_{(k-1)m}+
      \sum_{i=0}^{k-2}\mathcal{B}_{\Psi^\Delta,k}^{k-2-i}(\Psi^\Delta_{i+2}-
      \mathcal{B}_{\Psi,(i+2)m}^m)\mathcal{B}_{\Psi,(i+1)m}^m\mathcal{E}_{im},~k=2,\cdots,N_c
    \end{aligned}
  \right..
\end{align}
  It follows by the subsequent correction at C-points that
\begin{eqnarray*}
  \tilde{\mathcal{E}}_{km}=\mathcal{U}_{km}-\mathcal{V}_{km}-\mathcal{U}^{(2)}_k
  =\mathcal{B}_{\Psi,km}^m\mathcal{E}_{(k-1)m}-\mathcal{U}^{(2)}_k,~k=0,\cdots,N_c.
\end{eqnarray*}
  The desired result follows immediately by plugging \eqref{equ-03-09} into the above equation.
\end{proof}

It is important to note that $\Psi^\Delta_j$ is introduced to approximate the ideal coarse time-stepper $\mathcal{B}_{\Psi,jm}^m$.
An obvious and effective choice of $\Psi^\Delta_j$ is to re-discretize problem \eqref{chp-01-01}-\eqref{chp-01-03} on the coarse time-grid, i.e.,
\begin{eqnarray}\label{equ-03-02}
  \Psi^\Delta_j = \Big{[}\mathcal{M}_h+\frac{\tilde{t}_j-\tilde{t}_{j-1}}{2}(K_x\mathcal{A}^\beta_x+K_y\mathcal{A}^\gamma_y)\Big{]}^{-1}
  \Big{[}\mathcal{M}_h-\frac{\tilde{t}_j-\tilde{t}_{j-1}}{2}(K_x\mathcal{A}^\beta_x+K_y\mathcal{A}^\gamma_y)\Big{]}.
\end{eqnarray}
At this point, a coarse time-step is roughly as expensive to solve as a fine time-step.

Next we wish to establish the error reduction factor of the two-level version of Algorithm \ref{alg-03-01},
similar to the one in \cite{d-02}, 
based on the fact that $\Psi_j$ defined by \eqref{equ-03-01} and $\Psi^\Delta_j$ defined by \eqref{equ-03-02}
can be simultaneously diagonalized by a unitary matrix $\mathcal{X}=(\mathcal{X}_1,\cdots,\mathcal{X}_{|\mathcal{K}_h|})$, which satisfies
\begin{eqnarray}\label{equ-03-04}
  \mathcal{X}^*[\mathcal{M}_h^{-1}(K_x\mathcal{A}^\beta_x+K_y\mathcal{A}^\gamma_y)]\mathcal{X}
  =\textsf{diag}(\sigma_1,\cdots,\sigma_{|\mathcal{K}_h|}).
\end{eqnarray}
This will yield the following two lemmas.

\begin{lemma}\label{lem-03-02}
  The eigenvalue $\sigma_k$ in Eq. \eqref{equ-03-04} is real and positive for $k=1,\cdots,|\mathcal{K}_h|$.
\end{lemma}

\begin{proof}
  Note that the matrix $$\mathcal{M}_h^{-\frac{1}{2}}(K_x\mathcal{A}^\beta_x+K_y\mathcal{A}^\gamma_y)\mathcal{M}_h^{-\frac{1}{2}}$$
  is symmetric and similar to $\mathcal{M}_h^{-1}(K_x\mathcal{A}^\beta_x+K_y\mathcal{A}^\gamma_y)$,
  which, together with the positive definiteness of $B^t_\Omega(u,v)$ defined by \eqref{chp-02-07} (see reference \cite{b-02} for a proof),
  imply that the result is true.
\end{proof}

\begin{lemma}\label{lem-03-03}
  All time-grid propagators \eqref{equ-03-01} and \eqref{equ-03-02} are stable.
\end{lemma}

\begin{proof}
  For $j=1,\cdots,N$ and $k=1,\cdots,N_c$, let
\begin{eqnarray}\label{equ-03-06}
  \mathcal{X}^*\Psi_j\mathcal{X}=\textsf{diag}(\lambda_1^{(j)},\cdots,\lambda_{|\mathcal{K}_h|}^{(j)}),~~
  \mathcal{X}^*\Psi^\Delta_k\mathcal{X}=\textsf{diag}(\mu_1^{(k)},\cdots,\mu_{|\mathcal{K}_h|}^{(k)}).
\end{eqnarray}
  Then it is easy to verify that the following relations hold
\begin{eqnarray}\label{equ-03-05}
  \lambda_\omega^{(j)}=\frac{2-(t_j-t_{j-1})\sigma_\omega}{2+(t_j-t_{j-1})\sigma_\omega},~~
  \mu_\omega^{(k)}=\frac{2-(\tilde{t}_k-\tilde{t}_{k-1})\sigma_\omega}{2+(\tilde{t}_k-\tilde{t}_{k-1})\sigma_\omega}.
\end{eqnarray}
  By Lemma \ref{lem-03-02}, we conclude that $|\lambda_\omega^{(j)}|<1$ and $|\mu_\omega^{(k)}|<1$
  for all indices $j$, $k$ and eigenmodes $\omega$, and thus prove the lemma.
\end{proof}

Putting Lemmas \ref{lem-03-01}-\ref{lem-03-03} together and choosing $\Psi^\Delta_j$
as in \eqref{equ-03-02} give the following core results with respect to the Euclidean norm $\|\cdot\|_2$.

\begin{theorem}\label{thm-03-01}
  Let the error $\mathcal{E}_C=(\mathcal{E}^T_0,\mathcal{E}^T_m,\cdots,\mathcal{E}^T_N)^T$,
  time-grid propagators $\Psi_j$ and $\Psi^\Delta_k$ have eigenvalues $\lambda_\omega^{(j)}$ and $\mu_\omega^{(k)}$
  as in \eqref{equ-03-05}, respectively. Then the new error
  $\tilde{\mathcal{E}}_C=(\tilde{\mathcal{E}}^T_0,\tilde{\mathcal{E}}^T_m,\cdots,\tilde{\mathcal{E}}^T_N)^T$
  generated by the two-level version of Algorithm \ref{alg-03-01} holds
\begin{eqnarray}\label{equ-03-03}
  \|\tilde{\mathcal{E}}_C\|_2 \le \max_\omega\bigg{\{}(\lambda^\dagger_\omega)^m|(\lambda^\dagger_\omega)^m-\mu_\omega^\ddagger|
  \frac{1-(\mu_\omega^*)^{N_c-1}}{1-\mu_\omega^*}\bigg{\}} \|\mathcal{E}_C\|_2,
\end{eqnarray}
  where $\lambda^\dagger_\omega=\max\limits_j|\lambda_\omega^{(j)}|$,
  $\mu_\omega^\ddagger=\min\limits_j\mu_\omega^{(j)}$ and $\mu_\omega^*=\max\limits_j|\mu_\omega^{(j)}|$.
\end{theorem}

\begin{proof}
  Since the set of vectors $\{\mathcal{X}_\omega\}_{\omega=1}^{|\mathcal{K}_h|}$ is orthonormal,
  we can express $\mathcal{E}_{km}$ and $\tilde{\mathcal{E}}_{km}$ by eigenvector expansions
\begin{eqnarray}\label{equ-03-07}
  \mathcal{E}_{km}=\sum_{\omega=1}^{|\mathcal{K}_h|}e_\omega^{(km)}\mathcal{X}_\omega~\mbox{with}~
  e_\omega^{(km)}=\mathcal{X}^*_\omega\mathcal{E}_{km},~~
  \tilde{\mathcal{E}}_{km}=\sum_{\omega=1}^{|\mathcal{K}_h|}\tilde{e}_\omega^{(km)}\mathcal{X}_\omega~
  \mbox{with}~\tilde{e}_\omega^{(km)}=\mathcal{X}^*_\omega\tilde{\mathcal{E}}_{km}.
\end{eqnarray}
  According to Lemma \ref{lem-03-01} and \eqref{equ-03-06} in the proof of Lemma \ref{lem-03-03},
  $\tilde{e}_\omega^{(km)}$ can be reformulated as follows
\begin{eqnarray*}
  \tilde{e}_\omega^{(0)}=\tilde{e}_\omega^{(m)}=0,~~
  \tilde{e}_\omega^{(km)}=\sum_{i=0}^{k-2}\mathcal{Q}_i^k
  (\mathcal{S}_{i+1}-\mu_\omega^{(i+2)})\mathcal{S}_ie_\omega^{(im)},~k=2,\cdots,N_c,
\end{eqnarray*}
  alternatively, in matrix representation
\begin{eqnarray}\label{equ-03-08}
  \tilde{e}_\omega:=\left(
   \begin{array}{c}
    \tilde{e}_\omega^{(0)}  \\
    \tilde{e}_\omega^{(m)}  \\
          \vdots            \\
    \tilde{e}_\omega^{(N)}  \\
   \end{array}
  \right)=\mathcal{J}_\omega\left(
   \begin{array}{c}
    e_\omega^{(0)}  \\
    e_\omega^{(m)}  \\
          \vdots    \\
    e_\omega^{(N)}  \\
   \end{array}
  \right):=\mathcal{J}_\omega e_\omega
\end{eqnarray}
  with $\mathcal{J}_\omega$'s $(k,i)$-th entry: $\mathcal{J}^\omega_{k,i}=\mathcal{Q}_i^k
  (\mathcal{S}_{i+1}-\mu_\omega^{(i+2)})\mathcal{S}_i$ if $i\le k-2$, and $\mathcal{J}^\omega_{k,i}=0$ otherwise,
  where
\begin{eqnarray*}
  \mathcal{Q}_{k-2}^k=1,~~\mathcal{Q}_i^k=\prod_{j=3+i}^k\mu_\omega^{(j)},~i=0,1,\cdots,k-3;~~~
  \mathcal{S}_i=\prod_{j=im+1}^{(i+1)m}\lambda_\omega^{(j)}.
\end{eqnarray*}
  Utilize H\"{o}lder's inequality $\|\mathcal{J}_\omega\|_2\le\sqrt{\|\mathcal{J}_\omega\|_1\|\mathcal{J}_\omega\|_\infty}$
  to obtain an estimate
\begin{eqnarray*}
  \|\mathcal{J}_\omega\|_2\le(\lambda^\dagger_\omega)^m|(\lambda^\dagger_\omega)^m-\mu_\omega^\ddagger|
  \frac{1-(\mu_\omega^*)^{N_c-1}}{1-\mu_\omega^*}.
\end{eqnarray*}
  Combining the above inequality, Eq. \eqref{equ-03-07} and Eq. \eqref{equ-03-08}, it can be seen that
\begin{eqnarray*}
  \|\tilde{\mathcal{E}}_C\|_2^2&=&\sum_{k=0}^{N_c}\|\tilde{\mathcal{E}}_k\|^2=
  \sum_{k=0}^{N_c}\sum_{\omega=1}^{|\mathcal{K}_h|}|\tilde{e}_\omega^{(km)}|^2=
  \sum_{\omega=1}^{|\mathcal{K}_h|}\|\mathcal{J}_\omega e_\omega\|_2^2\\
  &\le&\max_\omega\|\mathcal{J}_\omega\|_2^2\sum_{\omega=1}^{|\mathcal{K}_h|}\|e_\omega\|_2^2
  =(\max_\omega\|\mathcal{J}_\omega\|_2)^2\|\mathcal{E}_C\|_2^2,
\end{eqnarray*}
  which leads to the inequality \eqref{equ-03-03}.
\end{proof}

\begin{remark}
  It is straightforward to verify that Lemma \ref{lem-03-01} and Theorem \ref{thm-03-01}
  are generalizations of theoretical results described in \cite{d-02}.
\end{remark}

\begin{theorem}\label{thm-03-02}
  The inequality \eqref{equ-03-03} provides an upper bound on the ratio of two successive fine-grid residual norms, i.e.,
\begin{eqnarray}\label{equ-03-11}
  \frac{\|r_{l+1}\|_2}{\|r_l\|_2}\le
  \max_\omega\bigg{\{}(\lambda^\dagger_\omega)^m|(\lambda^\dagger_\omega)^m-\mu_\omega^\ddagger|
  \frac{1-(\mu_\omega^*)^{N_c-1}}{1-\mu_\omega^*}\bigg{\}},
\end{eqnarray}
  where $r_l$ is the residual vector achieved at the $l$-th step of the two-level version of Algorithm \ref{alg-03-01}.
\end{theorem}

\begin{proof}
  Rewriting $\mathcal{A}$ in Eq. \eqref{ode-03} and $\mathcal{P}^{(0)}$ in block form regarding the given C/F splitting
\begin{eqnarray*}
\mathcal{A}=\begin{bmatrix}
\mathcal{A}_{FF} & \mathcal{A}_{FC} \\
\mathcal{A}_{CF} & \mathcal{A}_{CC}
\end{bmatrix},~~
\mathcal{P}^{(0)}=\begin{bmatrix}
-\mathcal{A}_{FF}^{-1}\mathcal{A}_{FC} \\
I_C
\end{bmatrix},
\end{eqnarray*}
  and letting $\tilde{\mathcal{E}}_l=((\tilde{\mathcal{E}}_F^{(l)})^T,(\tilde{\mathcal{E}}_C^{(l)})^T)^T$
  be the fine-grid error obtained after $l$ steps of the two-level version of Algorithm \ref{alg-03-01},
  yields the corresponding fine-grid residual norm
\begin{eqnarray}\label{equ-03-10}
  \|r_l\|_2=\|\mathcal{A}\tilde{\mathcal{E}}_l\|_2=
  \|\mathcal{A}\mathcal{P}^{(0)}\tilde{\mathcal{E}}_C^{(l)}\|_2=
  \|\mathcal{A}_\Delta\tilde{\mathcal{E}}_C^{(l)}\|_2,
\end{eqnarray}
  where the second equality uses the fact that the coarse-grid correction at F-points
  is equivalent to an F-relaxation over $\mathcal{A}$, and 
\begin{eqnarray*}
\mathcal{A}_\Delta=\begin{bmatrix}
  I &        &        &   \\
  -\mathcal{B}_{\Psi,m}^m & I &        &   \\
        & \ddots & \ddots &   \\
        &        &  -\mathcal{B}_{\Psi,N}^m & I
\end{bmatrix}
=\mathcal{A}_{CC}-\mathcal{A}_{CF}\mathcal{A}_{FF}^{-1}\mathcal{A}_{FC}.
\end{eqnarray*}
  Then, we have the following relation from Eq. \eqref{equ-03-10} and the proof of Theorem \ref{thm-03-01}
\begin{eqnarray*}
  \frac{\|r_{l+1}\|_2}{\|r_l\|_2}&=&\frac{\|\mathcal{A}_\Delta\tilde{\mathcal{E}}_C^{(l+1)}\|_2}{\|\mathcal{A}_\Delta
  \tilde{\mathcal{E}}_C^{(l)}\|_2}\\ &\le& \max_\omega \|\mathcal{A}_\Delta \mathcal{J}_\omega \mathcal{A}_\Delta^{-1}\|_2\\&=&
  \max_\omega\bigg{\{}(\lambda^\dagger_\omega)^m|(\lambda^\dagger_\omega)^m-\mu_\omega^\ddagger|
  \frac{1-(\mu_\omega^*)^{N_c-1}}{1-\mu_\omega^*}\bigg{\}},
\end{eqnarray*}
  where the last equality stems from the fact that $\mathcal{A}_\Delta$ and $\mathcal{J}_\omega$ commute.
  This completes the proof.
\end{proof}

\section{Test problems and numerical results}
\label{sec4}

Within the section, we illustrate the convergence behavior of our space-time FE numerical scheme \eqref{chp-02-03},
present theoretical confirmations and examine the performance evaluation of Algorithm \ref{alg-03-01}
in weak and strong parallel scaling studies, where includes the accuracy and generality of our analysis. 
Numerical experiments are performed on a 64-bit Linux cluster consisting of 32 compute nodes,
with sixteen 2.6 gigahertz Intel Xeon cores and 20 megabytes cache per compute node.
In tables below, columns labeled $\|error\|_0$ show error norms $\|u(x,y,T)-u_{h\tau}(x,y,T)\|_{L^2(\Omega)}$,
\emph{conv. rate} denote the convergence rates, \emph{T.U.B.} represent the theoretical upper bound governed by the inequality \eqref{equ-03-11}
and \emph{np} are the number of processors.

\subsection{Convergence behavior test}

\begin{example}\label{cex-04-01}
  Consider problem \eqref{chp-01-01}-\eqref{chp-01-03} with $T=1.0$,
  $\Omega=(0,1)\times(0,1)$, the initial data $\psi_0(x,y)=10(x-x^2)^2(y-y^2)^2$ and the source term
\begin{align*}
  & f(x,y,t)=-10e^{-t}(x-x^2)^2(y-y^2)^2\\
  &\quad+\frac{10K_xe^{-t}(y-y^2)^2}{\cos(\beta\pi)}
  \Big{[}\frac{x^{2-2\beta}+(1-x)^{2-2\beta}}{\Gamma(3-2\beta)}
  -6\frac{x^{3-2\beta}+(1-x)^{3-2\beta}}{\Gamma(4-2\beta)}+12\frac{x^{4-2\beta}+(1-x)^{4-2\beta}}{\Gamma(5-2\beta)}\Big{]}\\
  &\quad+\frac{10K_ye^{-t}(x-x^2)^2}{\cos(\gamma\pi)}\Big{[}\frac{y^{2-2\gamma}+(1-y)^{2-2\gamma}}{\Gamma(3-2\gamma)}
  -6\frac{y^{3-2\gamma}+(1-y)^{3-2\gamma}}{\Gamma(4-2\gamma)}+12\frac{y^{4-2\gamma}+(1-y)^{4-2\gamma}}{\Gamma(5-2\gamma)}\Big{]}.
\end{align*}
  Its exact solution is given by $u(x,y,t)=10e^{-t}(x-x^2)^2(y-y^2)^2$.
\end{example}

We start by the case of uniform temporal partitions. Tables \ref{ctp-04-01}-\ref{ctp-04-02}
are provided to address the $\mathcal{O}(h^2+\tau^2)$ error bound of space-time FE approximations
with different choices of $\beta$, $\gamma$ for two specific cases: $\tau=h$ and $\tau=h^3$,
where $\tau$ and $h$ respectively denote the uniform time and space step sizes. Fig. \ref{fig-04-01} gives the surface plots of
exact and numerical solutions in the case where $h=\tau=1/32$, $\beta=0.95$, $\gamma=0.65$, $K_x=2.0$ and $K_y=0.5$,
where a good agreement can be exploited in these two solutions.

\begin{table}[htbp]
\small
\centering\caption{Convergence behaviors with different choices of $\beta$, $\gamma$ for $\tau=h$, $K_x=2.0$ and $K_y=0.5$.}\label{ctp-04-01}\vskip 0.1cm
\begin{tabular}{||c|c|c|c|c|c|c|c|c|c|c|c|c|c|c|c|c|c|c|}\hline
\multirow{2}{*}{$M_\beta=M_\gamma$}&\multirow{2}{*}{$N$}&\multicolumn{2}{c|}{$\beta=0.6$, $\gamma=0.7$}
&\multicolumn{2}{c|}{$\beta=0.95$, $\gamma=0.65$}&\multicolumn{2}{c||}{$\beta=\gamma=0.8$} \\ \cline{3-8}
~&~& $\|error\|_0$ & \emph{conv. rate} & $\|error\|_0$ & \emph{conv. rate} & $\|error\|_0$ & \multicolumn{1}{c||}{\emph{conv. rate}} \\ \hline

  4 &  4 & 9.1102E-4 &   -   & 1.0732E-3 &   -   & 9.7198E-4 & \multicolumn{1}{c||}{  -  } \\ \hline
  8 &  8 & 2.1317E-4 & 2.095 & 3.1841E-4 & 1.753 & 2.4527E-4 & \multicolumn{1}{c||}{1.987} \\ \hline
 16 & 16 & 4.8931E-5 & 2.123 & 7.9071E-5 & 2.010 & 5.6258E-5 & \multicolumn{1}{c||}{2.124} \\ \hline
 32 & 32 & 1.1402E-5 & 2.101 & 1.9200E-5 & 2.042 & 1.3269E-5 & \multicolumn{1}{c||}{2.084} \\ \hline

\end{tabular}
\end{table}

\begin{table}[htbp]
\small
\centering\caption{Convergence behaviors with different choices of $\beta$, $\gamma$ for $\tau=h^3$, $K_x=3.0$ and $K_y=7.5$.}\label{ctp-04-02}\vskip 0.1cm
\begin{tabular}{||c|c|c|c|c|c|c|c|c|c|c|c|c|c|c|c|c|c|c|}\hline
\multirow{2}{*}{$M_\beta=M_\gamma$}&\multirow{2}{*}{$N$}&\multicolumn{2}{c|}{$\beta=0.6$, $\gamma=0.7$}
&\multicolumn{2}{c|}{$\beta=0.95$, $\gamma=0.65$}&\multicolumn{2}{c||}{$\beta=\gamma=0.8$} \\ \cline{3-8}
~&~& $\|error\|_0$ & \emph{conv. rate} & $\|error\|_0$ & \emph{conv. rate} & $\|error\|_0$ & \multicolumn{1}{c||}{\emph{conv. rate}} \\ \hline

  4 &     64 & 9.0446E-4 &   -   & 9.9374E-4 &   -   & 9.5859E-4 & \multicolumn{1}{c||}{  -  } \\ \hline
  8 &    512 & 2.1566E-4 & 2.068 & 2.5964E-4 & 1.936 & 2.3889E-4 & \multicolumn{1}{c||}{2.005} \\ \hline
 16 &   4096 & 4.9575E-5 & 2.121 & 6.2424E-5 & 2.056 & 5.5742E-5 & \multicolumn{1}{c||}{2.100} \\ \hline
 32 &  32768 & 1.1607E-5 & 2.095 & 1.5211E-5 & 2.037 & 1.3240E-5 & \multicolumn{1}{c||}{2.074} \\ \hline

\end{tabular}
\end{table}

\begin{figure}[htp]
\begin{minipage}[t]{0.5\linewidth}
 \centering
 \subfigure[The exact solution]{%
 \label{fig-04-01-a}
 \includegraphics[width=3in]{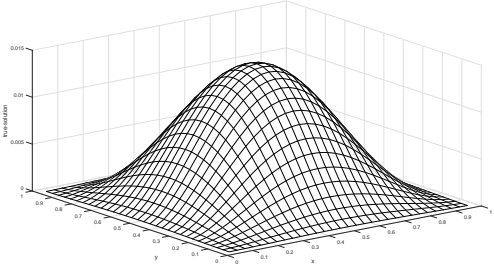}}
\end{minipage}%
\begin{minipage}[t]{0.5\linewidth}
 \centering
 \subfigure[The numerical solution]{%
 \label{fig-04-01-b}
 \includegraphics[width=3in]{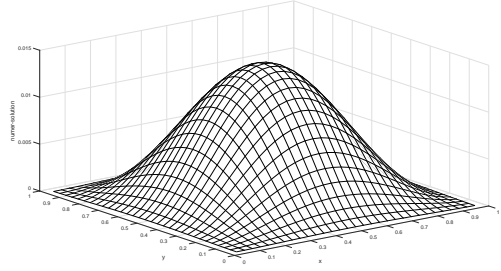}}
\end{minipage}
\caption{Accuracy test for Example \ref{cex-04-01} at $t=1$, $h=\tau=1/32$, $\beta=0.95$, $\gamma=0.65$, $K_x=2.0$, $K_y=0.5$.}\label{fig-04-01}
\end{figure}

For the nonuniform temporal case, we only consider the piecewise uniform partition:
subinterval uniform partition with $N/4$ knots on $[0,\sigma]$ and $[1-\sigma,1]$, while $N/2$ knots on $[\sigma,1-\sigma]$,
where $\sigma$ is the transition point defined by $\sigma=2\varepsilon\ln N$ \cite{s-03}.
Tables \ref{ctp-04-03}-\ref{ctp-04-04} enumerate error norms and convergence rates for $N=1/h$ and $N=1/h^3$.
It can easily be seen that the saturation error order carries over.

\begin{table}[htbp]
\small
\centering\caption{Convergence behaviors with different choices of
$\beta$, $\gamma$ for $N=1/h$, $\varepsilon=2^{-6}$, $K_x=3.0$ and $K_y=7.5$.}\label{ctp-04-03}\vskip 0.1cm
\begin{tabular}{||c|c|c|c|c|c|c|c|c|c|c|c|c|c|c|c|c|c|c|}\hline
\multirow{2}{*}{$M_\beta=M_\gamma$}&\multirow{2}{*}{$N$}&\multicolumn{2}{c|}{$\beta=0.6$, $\gamma=0.7$}
&\multicolumn{2}{c|}{$\beta=0.95$, $\gamma=0.65$}&\multicolumn{2}{c||}{$\beta=\gamma=0.8$} \\ \cline{3-8}
~&~& $\|error\|_0$ & \emph{conv. rate} & $\|error\|_0$ & \emph{conv. rate} & $\|error\|_0$ & \multicolumn{1}{c||}{\emph{conv. rate}} \\ \hline

  4 &  4 & 9.3174E-4 &   -   & 9.8898E-4 &   -   & 9.7527E-4 & \multicolumn{1}{c||}{  -  } \\ \hline
  8 &  8 & 2.2117E-4 & 2.052 & 2.6382E-4 & 1.936 & 2.5103E-4 & \multicolumn{1}{c||}{1.971} \\ \hline
 16 & 16 & 4.9944E-5 & 2.104 & 6.3043E-5 & 2.046 & 5.7106E-5 & \multicolumn{1}{c||}{2.097} \\ \hline
 32 & 32 & 1.1615E-5 & 2.074 & 1.5298E-5 & 2.030 & 1.3368E-5 & \multicolumn{1}{c||}{2.067} \\ \hline

\end{tabular}
\end{table}

\begin{table}[htbp]
\small
\centering\caption{Convergence behaviors with different choices of
$\beta$, $\gamma$ for $N=1/h^3$, $\varepsilon=2^{-8}$, $K_x=2.0$ and $K_y=0.5$.}\label{ctp-04-04}\vskip 0.1cm
\begin{tabular}{||c|c|c|c|c|c|c|c|c|c|c|c|c|c|c|c|c|c|c|}\hline
\multirow{2}{*}{$M_\beta=M_\gamma$}&\multirow{2}{*}{$N$}&\multicolumn{2}{c|}{$\beta=0.6$, $\gamma=0.7$}
&\multicolumn{2}{c|}{$\beta=0.95$, $\gamma=0.65$}&\multicolumn{2}{c||}{$\beta=\gamma=0.8$} \\ \cline{3-8}
~&~& $\|error\|_0$ & \emph{conv. rate} & $\|error\|_0$ & \emph{conv. rate} & $\|error\|_0$ & \multicolumn{1}{c||}{\emph{conv. rate}} \\\hline

  4 &    64 & 8.9572E-4 &   -   & 1.1194E-3 &   -   & 9.6413E-4 & \multicolumn{1}{c||}{  -  } \\ \hline
  8 &   512 & 2.1320E-4 & 2.050 & 3.1590E-4 & 1.882 & 2.3996E-4 & \multicolumn{1}{c||}{2.004} \\ \hline
 16 &  4096 & 4.9216E-5 & 2.081 & 7.7853E-5 & 2.014 & 5.5905E-5 & \multicolumn{1}{c||}{2.072} \\ \hline
 32 & 32768 & 1.1480E-5 & 2.071 & 1.8941E-5 & 2.027 & 1.3271E-5 & \multicolumn{1}{c||}{2.052} \\ \hline

\end{tabular}
\end{table}

\subsection{Comparisons on convergence factors of the two-level MGRIT}

\begin{example}\label{cex-04-02}
  We use the same Example \ref{cex-04-01}. 
\end{example}

The objective of this example is to measure estimations from Theorems \ref{thm-03-01}-\ref{thm-03-02}
and asymptotic convergence rates over the final five MGRIT iterations. We can observe from
Tables \ref{ctp-04-05}-\ref{ctp-04-06} with different $\beta$, $\gamma$, $m$ for $N=1/h^2$, $N=1/h^3$
that, in all cases, the observed results are very close to our theoretical estimates,
indicating that Theorems \ref{thm-03-01}-\ref{thm-03-02} offers good bounds for two-dimensional SFDEs on these two temporal meshes.

\begin{table}[htbp]
\small
\centering\caption{Asymptotic convergence factors of uniform temporal meshes
for $K_x=2.0$, $K_y=0.5$.}\label{ctp-04-05}\vskip 0.1cm
\begin{tabular}{||c|c|c|c|c|c|c|c|c|c|c|c|c|c|c|c|c|c|c|}\hline
\multirow{2}{*}{$M_\beta=M_\gamma$}&\multirow{2}{*}{$N$}&\multicolumn{4}{c|}{$\beta=0.6$, $\gamma=0.7$}
&\multicolumn{4}{c||}{$\beta=0.95$, $\gamma=0.65$} \\ \cline{3-10}
~&~& $m=2$ & \emph{T.U.B.} & $m=16$ & \emph{T.U.B.} & $m=2$ & \emph{T.U.B.} & $m=16$ & \multicolumn{1}{c||}{\emph{T.U.B.}} \\ \hline \hline

 16 &  256 & 0.0094 & \multirow{3}{*}{0.0122} & 0.0081 & \multirow{3}{*}{0.0118}
           & 0.0078 & \multirow{3}{*}{0.0115} & 0.0098 & \multicolumn{1}{c||}{\multirow{3}{*}{0.0135}} \\ \cline{1-3} \cline{5-5} \cline{7-7} \cline{9-9}
 32 & 1024 & 0.0103 &  & 0.0096 &  & 0.0091 &  & 0.0110 & \multicolumn{1}{c||}{} \\ \cline{1-3} \cline{5-5} \cline{7-7} \cline{9-9}
 64 & 4096 & 0.0111 &  & 0.0105 &  & 0.0101 &  & 0.0122 & \multicolumn{1}{c||}{} \\ \hline \hline

  8 &   512 & 0.0134 & \multirow{3}{*}{0.0153} & 0.0163 & \multirow{3}{*}{0.0183}
            & 0.0137 & \multirow{3}{*}{0.0156} & 0.0188 & \multicolumn{1}{c||}{\multirow{3}{*}{0.0206}} \\ \cline{1-3} \cline{5-5} \cline{7-7} \cline{9-9}
 16 &  4096 & 0.0140 &  & 0.0169 &  & 0.0141 &  & 0.0191 & \multicolumn{1}{c||}{} \\ \cline{1-3} \cline{5-5} \cline{7-7} \cline{9-9}
 32 & 32768 & 0.0142 &  & 0.0174 &  & 0.0145 &  & 0.0195 & \multicolumn{1}{c||}{} \\ \hline

\end{tabular}
\end{table}

\begin{table}[htbp]
\small
\centering\caption{Asymptotic convergence factors of piecewise uniform temporal meshes
for $\varepsilon=2^{-6}$, $K_x=3.0$, $K_y=7.5$.}\label{ctp-04-06}\vskip 0.1cm
\begin{tabular}{||c|c|c|c|c|c|c|c|c|c|c|c|c|c|c|c|c|c|c|}\hline
\multirow{2}{*}{$M_\beta=M_\gamma$}&\multirow{2}{*}{$N$}&\multicolumn{4}{c|}{$\beta=0.6$, $\gamma=0.7$}
&\multicolumn{4}{c||}{$\beta=0.95$, $\gamma=0.65$} \\ \cline{3-10}
~&~& $m=2$ & \emph{T.U.B.} & $m=4$ & \emph{T.U.B.} & $m=4$ & \emph{T.U.B.} & $m=8$ & \multicolumn{1}{c||}{\emph{T.U.B.}} \\ \hline \hline

 16 &  256 & 0.0179 & \multirow{3}{*}{0.0285} & 0.0117 & \multirow{3}{*}{0.0168}
           & 0.0163 & \multirow{3}{*}{0.0218} & 0.0130 & \multicolumn{1}{c||}{\multirow{3}{*}{0.0182}} \\ \cline{1-3} \cline{5-5} \cline{7-7} \cline{9-9}
 32 & 1024 & 0.0213 &  & 0.0139 &  & 0.0192 &  & 0.0157 & \multicolumn{1}{c||}{} \\ \cline{1-3} \cline{5-5} \cline{7-7} \cline{9-9}
 64 & 4096 & 0.0236 &  & 0.0155 &  & 0.0207 &  & 0.0172 & \multicolumn{1}{c||}{} \\ \hline \hline

  8 &   512 & 0.0318 & \multirow{3}{*}{0.0554} & 0.0277 & \multirow{3}{*}{0.0450}
            & 0.0273 & \multirow{3}{*}{0.0329} & 0.0161 & \multicolumn{1}{c||}{\multirow{3}{*}{0.0214}} \\ \cline{1-3} \cline{5-5} \cline{7-7} \cline{9-9}
 16 &  4096 & 0.0370 &  & 0.0329 &  & 0.0291 &  & 0.0184 & \multicolumn{1}{c||}{} \\ \cline{1-3} \cline{5-5} \cline{7-7} \cline{9-9}
 32 & 32768 & 0.0412 &  & 0.0368 &  & 0.0308 &  & 0.0203 & \multicolumn{1}{c||}{} \\ \hline

\end{tabular}
\end{table}

\subsection{Parallel scaling results}

\begin{example}\label{cex-04-03}
  Consider problem \eqref{chp-01-01}-\eqref{chp-01-03} with $T=1.0$, $\Omega=(0,1)^2$ and the zero source term
  to illustrate that MGRIT algorithm afford good approximations to the exact solution.
\end{example}

The first numerical results are weak parallel scalabilities of the parareal (solid lines) and the truly multilevel MGRIT (dashed lines)
with the problem size per processor being $128^2\times256$, as depicted in Fig. \ref{fig-04-02}. Similar to integer order parabolic problems,
parareal is slightly faster than MGRIT only for small processor counts, but appears a strong growth in the compute time;
MGRIT is beneficial for its much better parallel scalability,
and the crossover point of MGRIT over parareal is about at 16 processors for this particular problem.
For 512 processors, the overall time-to-solution reduces from 649 and 668 seconds for parareal, respectively, to MGRIT timings of 519 and 527 seconds.

\begin{figure}[htp]
\begin{minipage}[t]{0.5\linewidth}
 \centering
 \subfigure[Uniform temporal mesh]{%
 \label{fig-04-02-a}
 \includegraphics[width=3in]{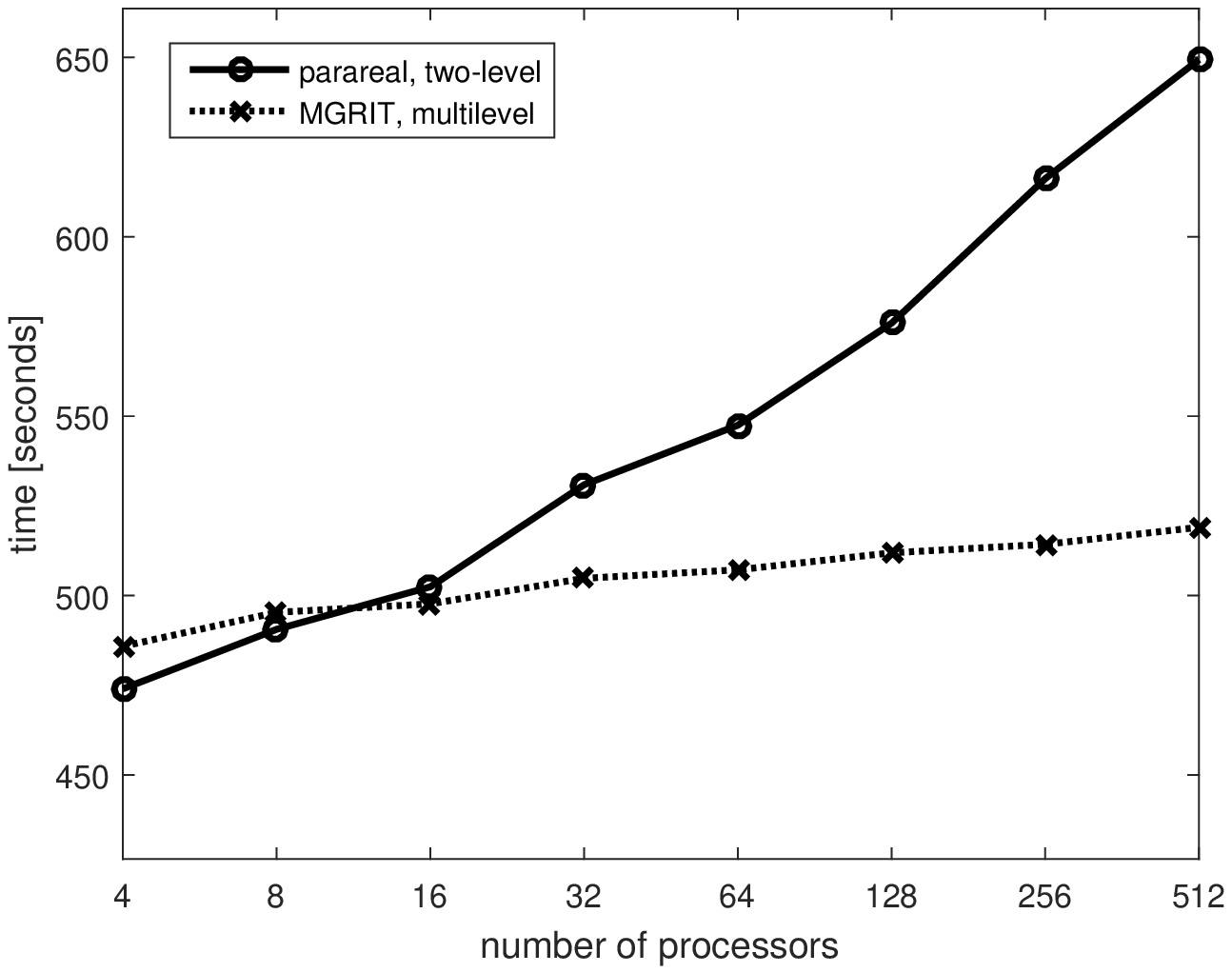}}
\end{minipage}%
\begin{minipage}[t]{0.5\linewidth}
 \centering
 \subfigure[Piecewise uniform temporal mesh]{%
 \label{fig-04-02-b}
 \includegraphics[width=3in]{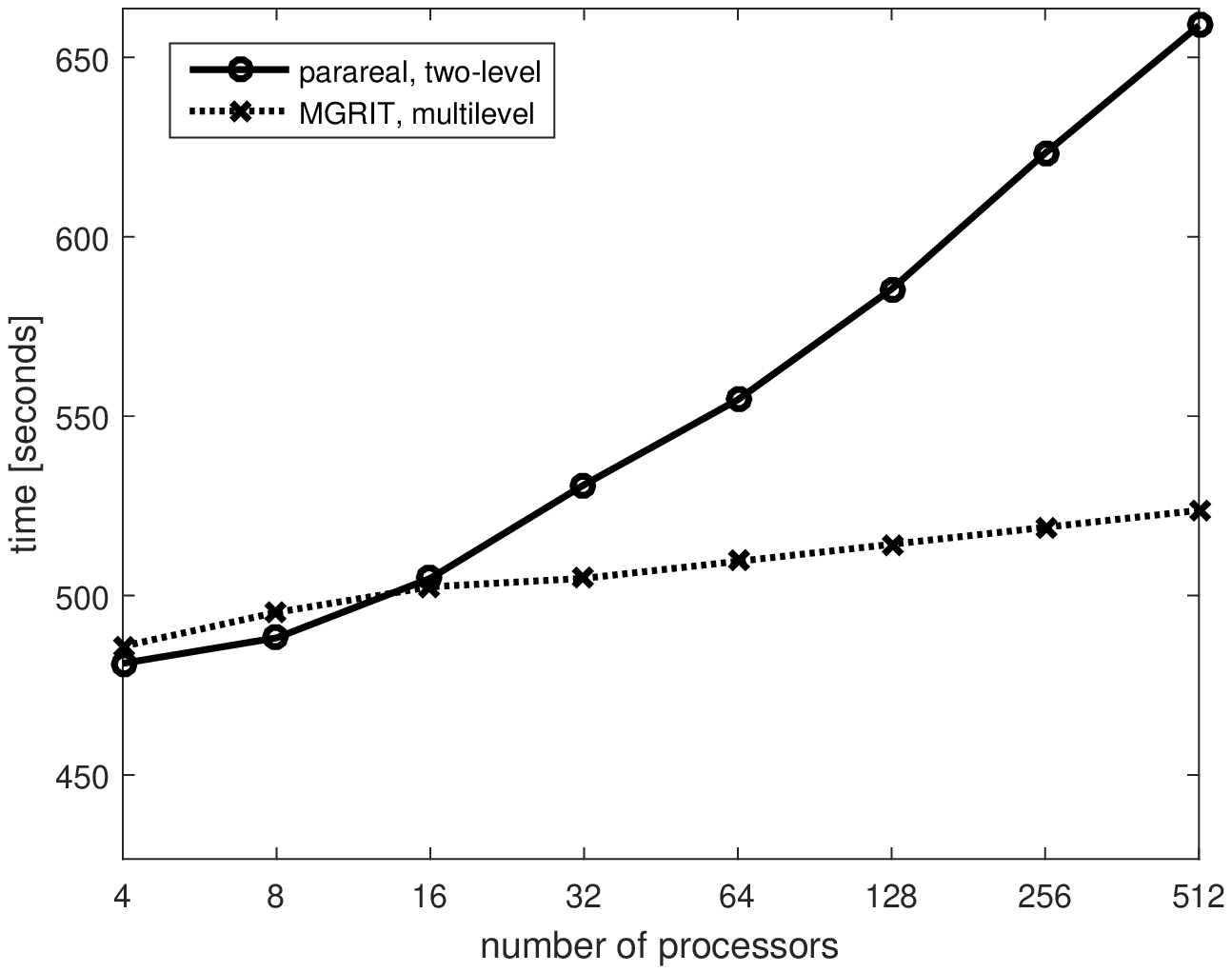}}
\end{minipage}
\caption{Comparisons on the overall time-to-solution between parareal and MGRIT in a weak scaling study.}\label{fig-04-02}
\end{figure}

The subsequent experiments are run for strong parallel scaling studies on a $128^2\times8192$ space-time grid
with an emphasis on comparing space-only parallelism (sequential time stepping) and space-time parallelism (MGRIT).
Here we utilize the factor-$m$ ($m=2,4,8$) coarsening strategies on all levels.
Fig. \ref{fig-04-03} illustrates comparisons on compute times with parareal and MGRIT both using 4 processors in spatial dimension,
because of its minimum overall time-to-solution in space-only parallelization.
The crossover point at which MGRIT becomes beneficial to use is about 32 processors, whereas 16 processors for parareal.
The time curves corresponding to the uniform temporal meshes look similar to those in Fig. \ref{fig-04-03}.
Tables \ref{ctp-04-07}-\ref{ctp-04-08} detail wall times and speedups of parareal and MGRIT,
where we measure the speedup relative to the wall time of sequential time-stepping with 4 processors.
Fig. \ref{fig-04-03-a} shows that compute times of parareal stagnate or increase slightly as of 256 processors.
In contrast, MGRIT is invariably optimistic to speed up computations.
As shown in Tables \ref{ctp-04-07}-\ref{ctp-04-08}, the best speedups are 2.046 and 4.646,
achieved respectively by processors of 128 for parareal and 512 for MGRIT both with the coarsening factor $m=8$.

\begin{figure}[htp]
\begin{minipage}[t]{0.5\linewidth}
 \centering
 \subfigure[parareal vs. time stepping]{%
 \label{fig-04-03-a}
 \includegraphics[width=3in]{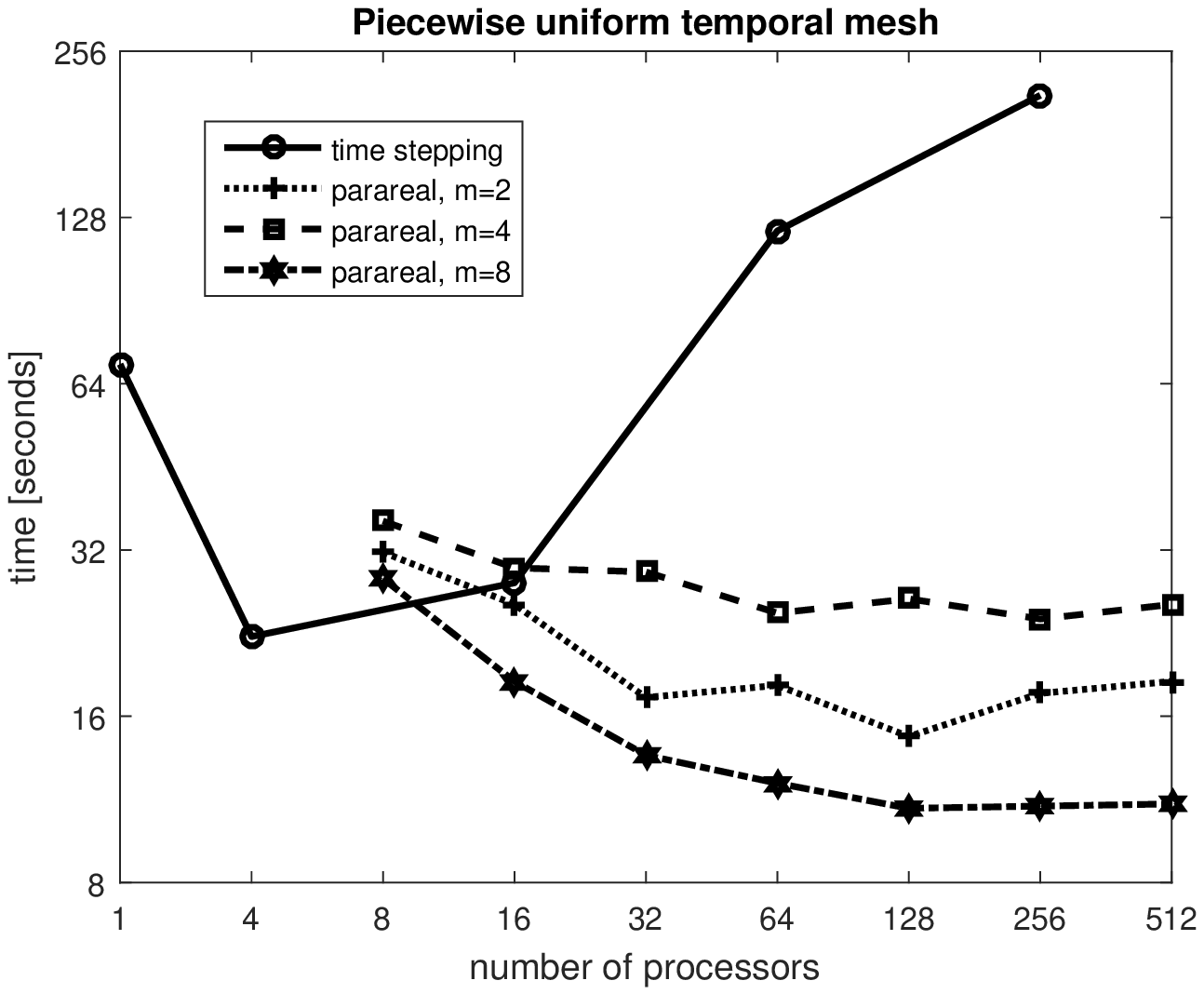}}
\end{minipage}%
\begin{minipage}[t]{0.5\linewidth}
 \centering
 \subfigure[MGRIT vs. time stepping]{%
 \label{fig-04-03-b}
 \includegraphics[width=3in]{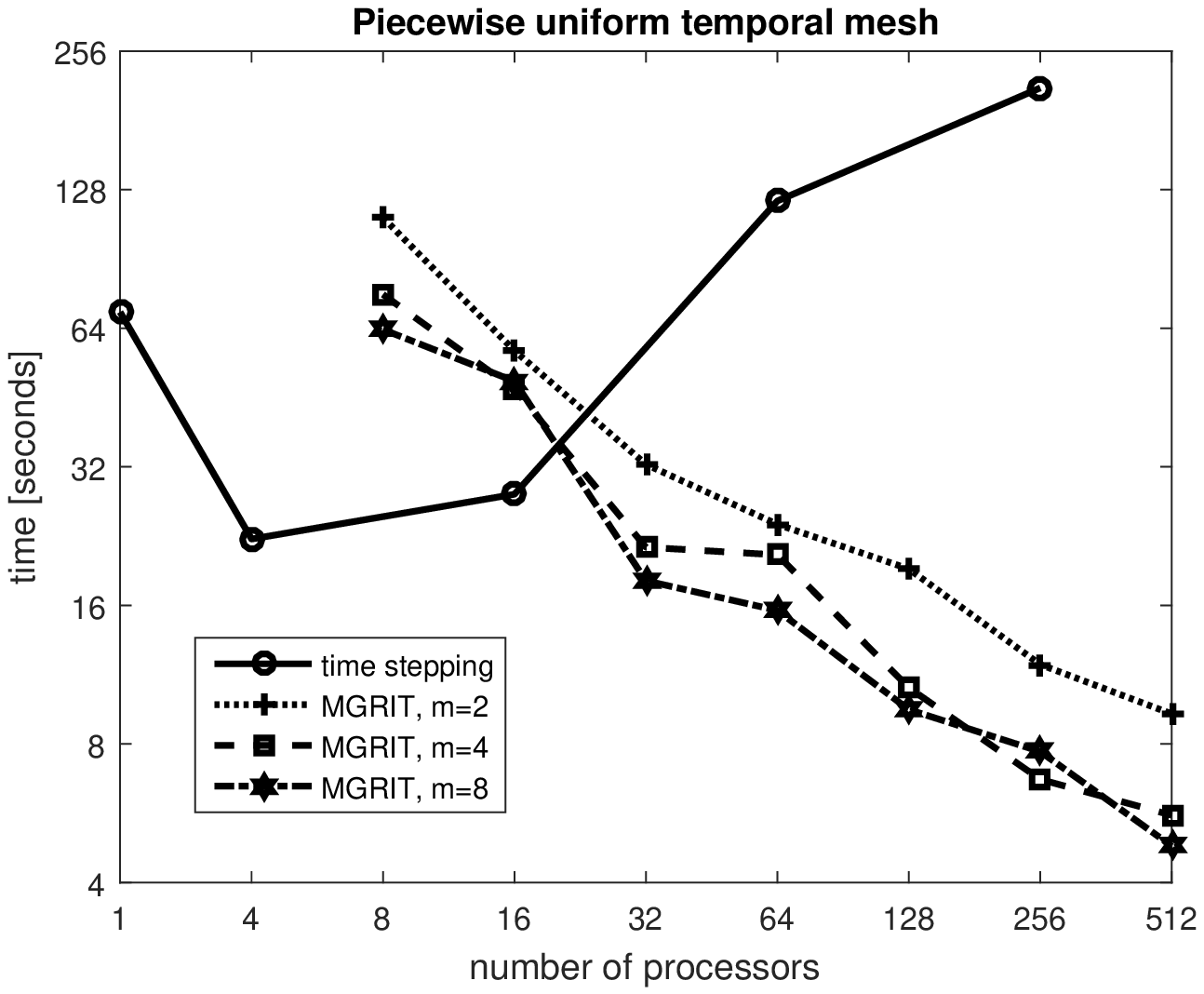}}
\end{minipage}
\caption{Comparisons on the compute time of parareal and MGRIT in a strong scaling study.}\label{fig-04-03}
\end{figure}

\begin{table}[htbp]
\small
\centering\caption{The wall time and speedup of parareal with different coarsening factor $m$.}\label{ctp-04-07}\vskip 0.1cm
\begin{tabular}{||c|c|c|c|c|c|c|c|c|c|c|c|c|c|c|c|c|c|c|}\hline
\multirow{2}{*}{\emph{np}}&\multicolumn{2}{c|}{parareal, $m=2$}
&\multicolumn{2}{c|}{parareal, $m=4$}&\multicolumn{2}{c||}{parareal, $m=8$} \\ \cline{2-7}
~& WTime & Speedup & WTime & Speedup & WTime & \multicolumn{1}{c||}{Speedup} \\ \hline

   8 & 36.1 & 0.618 & 31.7 & 0.703 & 28.4 & \multicolumn{1}{c||}{0.785} \\ \hline
  16 & 29.7 & 0.751 & 25.5 & 0.875 & 18.5 & \multicolumn{1}{c||}{1.205} \\ \hline
  32 & 29.2 & 0.764 & 17.3 & 1.289 & 13.6 & \multicolumn{1}{c||}{1.640} \\ \hline
  64 & 24.6 & 0.907 & 18.2 & 1.225 & 12.1 & \multicolumn{1}{c||}{1.843} \\ \hline
 128 & 26.1 & 0.854 & 14.7 & 1.517 & 10.9 & \multicolumn{1}{c||}{2.046} \\ \hline
 256 & 24.0 & 0.929 & 17.6 & 1.267 & 11.0 & \multicolumn{1}{c||}{2.027} \\ \hline
 512 & 25.5 & 0.875 & 18.5 & 1.205 & 11.1 & \multicolumn{1}{c||}{2.009} \\ \hline

\end{tabular}
\end{table}

\begin{table}[htbp]
\small
\centering\caption{The wall time and speedup of MGRIT with different coarsening factor $m$.}\label{ctp-04-08}\vskip 0.1cm
\begin{tabular}{||c|c|c|c|c|c|c|c|c|c|c|c|c|c|c|c|c|c|c|}\hline
\multirow{2}{*}{\emph{np}}&\multicolumn{2}{c|}{MGRIT, $m=2$}
&\multicolumn{2}{c|}{MGRIT, $m=4$}&\multicolumn{2}{c||}{MGRIT, $m=8$} \\ \cline{2-7}
~& WTime & Speedup & WTime & Speedup & WTime & \multicolumn{1}{c||}{Speedup} \\ \hline

   8 & 111.3 & 0.200 & 75.8 & 0.294 & 63.6 & \multicolumn{1}{c||}{0.351} \\ \hline
  16 &  57.2 & 0.390 & 47.2 & 0.472 & 49.1 & \multicolumn{1}{c||}{0.454} \\ \hline
  32 &  32.4 & 0.688 & 21.4 & 1.042 & 18.1 & \multicolumn{1}{c||}{1.232} \\ \hline
  64 &  24.0 & 0.929 & 20.6 & 1.083 & 15.6 & \multicolumn{1}{c||}{1.429} \\ \hline
 128 &  19.2 & 1.161 & 10.5 & 2.124 &  9.5 & \multicolumn{1}{c||}{2.347} \\ \hline
 256 &  11.9 & 1.874 &  6.7 & 3.328 &  7.7 & \multicolumn{1}{c||}{2.896} \\ \hline
 512 &   9.3 & 2.398 &  5.6 & 3.982 &  4.8 & \multicolumn{1}{c||}{4.646} \\ \hline

\end{tabular}
\end{table}

\section{Conclusion}
\label{sec5}

In this paper we added temporal parallelism on the top of the existing spatial parallelism
to speed up space-time FE discretizations of two-dimensional SFDEs, 
and generalized the two-level convergence analysis for MGRIT to the linear FE discretization of the temporal derivative
and time-dependent propagators.
Numerical results include the saturation error order of the discretization to uniform and piecewise uniform temporal partitions,
quantitatively correct predictions on convergence factors of the two-level MGRIT,
and considerable advantages in overall compute times over parareal and sequential time stepping approaches.
The introduced MGRIT method can be extended readily to implicit Runge-Kutta discretizations, three-dimensional SFDEs and integer order parabolic problems.
Our future work will be immersed in constructions of MGRIT to multidimensional nonlinear SFDEs and time fractional problems.

\section*{Acknowledgments}

This work is under auspices of National Natural Science Foundation of China (41474103, 11571293, 11601460, 11601462, 11771368),
Excellent Youth Foundation of Hunan Province of China (2018JJ1042),
Hunan Provincial Natural Science Foundation of China (2018JJ3494),
General Project of Hunan Provincial Education Department of China (16C1540, 17C1527).
The authors would like to thank Dr. M.~H.~Liu from Xiangtan University for many useful discussions.


\end{document}